\documentclass[12pt]{article}
\usepackage{a4wide}
\usepackage{amsfonts}
\usepackage{amssymb, amsthm}
\usepackage{amsmath}
\usepackage[cp1251]{inputenc}

\def\codim{{\rm{codim\,}}}

\def\Aug{{\rm{Aug\,}}}

\newtheorem{theorem}{Theorem}
\newtheorem{lemma}{Lemma}
\newtheorem*{corollary}{Corollary}
\newtheorem*{remark}{Remark}
\newtheorem*{conjecture}{Conjecture}

\usepackage{multirow}
\newcommand{\vdotrows}[2][-15pt]{%
  \multirow{#2}{*}{%
    \vbox {%
      \baselineskip = \dimexpr 2pt\relax 
      \multiply\baselineskip by #2\relax
      \advance\baselineskip by 2pt\relax
      \lineskiplimit = 0pt\relax
      \kern 6pt\relax
      \vskip #1\relax%
      \hbox {.}\hbox {.}\hbox {.}}%
  }%
}

\begin{document}
\title{Combinatorial results implied by many zero divisors in a group ring}
\author{Fedor Petrov}
\maketitle

\begin{abstract} It has been recently proved \cite{CLP,EG} 
that for a group $G=G_0^n$, where $G_0\ne \{1,-1\}^m$ is a fixed finite Abelian group
and $n$ is large, any subset $A$ without 3-progressions (triples $x,y,z$
of different elements with $xy=z^2$) contains at most $|G|^{1-c}$
elements, where $c>0$ is a constant depending only on $G_0$. This is known to be
false when $G$ is, say, large cyclic group, see current records in \cite{OB}. 
The aim of this note is to show
that algebraic property which corresponds
to this difference is the following: in the first case a group algebra 
$\mathbb{F}[G]$ over suitable field $\mathbb{F}$ contains a 
subspace $X$ with codimension at most $|X|^{1-c}$ such that $X^3=0$. We discuss which
bounds are obtained for finite Abelian $p$-groups and for some matrix $p$-groups:
Heisenberg group over $\mathbb{F}_p$ and the unitriangular group over
$\mathbb{F}_p$. Also we show how the method works for further generalizations by 
Kleinberg--Sawin--Speyer and Ellenberg. 
\end{abstract}

Let $C_N$ denote a cyclic group of order $N$ in multiplicative notation. Denote $\kappa_N=\min_{x>0} x^{-(N-1)/3}(1+x+\dots+x^{N-1})$.
The number of points in $\{0,1,\dots,N-1\}^n$ with sum of coordinates at most $n(N-1)/3$ grows as $(c_N+o(1))^n$ for large $n$.

In a recent paper \cite{CLP} Ernie Croot, Vsevolod Lev, and Peter Pach proved by a clever combination of
\textit{polynomial method} (in spirit of Alon's Combinatorial Nullstellensatz \cite{A}), \textit{linear algebraic} 
dimension reasoning and \textit{law of large numbers} 
that any subset $A$ of a group $G=C_4^n$ containing
more than $\kappa_4^n$ elements, has three distinct elements $a,b,c$ for which $ab=c^2$ (they used another
formula for the same constant $\kappa_4=3.61\dots$, but I prefer this unified formula,
which appears also below in Theorem \ref{Abelian}). 
It have been soon
observed that their method, being slightly modified (and even simplified! They started
from combinatorially harder, ``ramified'' problem), 
works for other groups, like $C_p^n$
for prime $p$ \cite{EG}, where we may get the same type bound $\kappa_p^n$ for 
3-progression-free sets. It also works
for other combinatorial problems of the same spirit, like Sunflower 
problems of Erd\"os--Szemeredi and Erd\"os -- Rado \cite{NS}. 
Robert Kleinberg, Will Sawin and David Speyer \cite{KSS}
observed that \cite{EG} actually contains the
bound for additive 3-matchings (see the definition below). Important
feature of this generalization is that the constant $\kappa_p$ in this question
is proved to be sharp (the proof was finished independently by Sergey Norin
and Luce Pebody \cite{No,Pe}). Next, Jordan Ellenberg proposed further generalization
proving that the sumset of any two sets may be covered by sumsets of small subsets.

It is natural that polynomial method works well for the groups which are contained either in additive or multiplicative
group of a field. For more general groups sometimes it may be successfully replaced by considering generating functions in
group rings, historically this approach appeared even earlier than Combinatorial Nullstellensatz: 
John E. Olson \cite{O}
used it for computing Davenport constant of finite abelian $p$-groups.

\section{Example: Cauchy--Davenport Theorem}

For illustrating a parallelism between Combinatorial Nullstellensatz method
and group rings method, we give two proofs of the 
Cauchy--Davenport theorem, which is convenient to
formulate as follows:

\textit{If $A,B,C$ are non-empty subsets of a cyclic group $G=C_p$ of prime order $p$ and $|A|+|B|+|C|=p+2$, then $ABC:=\{abc | a\in A,b\in B,c\in C\}=G$. }

1) Polynomial proof (see \cite{A}, where this proof is expressed bit differently). 
We identify $C_p$ with an additive group of $\mathbb{F}_p$. 
Assume that some element $h$ does not belong to $A+B+C$. Then a polynomial
$\varphi_h(x,y,z)=(x+y+z-h)^{p-1}-1$ takes zero values on
$A\times B\times C$. Let $u,v,w$ be $\mathbb{F}_p$-valued functions on $A,B,C$
to be specified later. Consider the following sum
$$
\sum_{x\in A,y\in B,z\in C} u(x)v(y)h(z) \varphi_h(x,y,z)=0.
$$
Now change the order of summation: sum up by monomials
of $\varphi_h(x,y,z)$. For each monomial $x^{\alpha} y^\beta z^\gamma$ we have
\begin{equation}\label{cdp}
\sum_{x\in A,y\in B,z\in C} u(x)v(y)h(z) x^{\alpha}y^{\beta}z^{\gamma} =
\sum_{x\in A} u(x) x^{\alpha}\cdot \sum_{y\in B} v(y) y^{\beta}\cdot \sum_{z\in C} w(z) z^{\gamma}.
\end{equation}
Now we are ready to say what we require from functions $u,v,w$. Assume that
$$
\sum_{x\in A} u(x) x^{\alpha}=\begin{cases}0,&\text{if}\,\,\alpha\leqslant |A|-2\\
1,&\text{if}\,\,\alpha=|A|-1,\end{cases}
$$
and analogous conditions hold for $\beta$ and $\gamma$. Such functions do exist:
linear systems for their values have Vandermonde matrices, which are non-degenerated. 
Then for any monomial 
$x^{\alpha} y^\beta z^\gamma$ in $\varphi_h$ either $\alpha=|A|-1$, $\beta=|B|-1$,
$\gamma=|C|-1$ or there is a zero multiple in right hand side
of \eqref{cdp}. But $\varphi_h$ has this monomial with non-zero coefficient
$\frac{(p-1)!}{(|A|-1)!(|B|-1)!(|C|-1)!}$. Hence the total sum is not equal to 0, a contradiction.

2) Group ring proof. Now we use a multiplicative notation for $G=C_p$, fix
a generator $g_0$ in $C_p$ and denote $\tau=g_0-1$ in the group algebra  $\mathbb{F}_p[G]$ (hereafter considering group
algebras we identify the unit elements of the group
and the field). We have $\tau^p=0$.
The ring $\mathbb{F}_p[G]$ is filtrated by the powers of its
augmentation ideal $\Aug(G)=\tau\cdot \mathbb{F}_p[G]$.

Let $u,v,w$ be $\mathbb{F}_p$-valued functions on $A,B,C$ respectively. Consider the following
product
$$
\left(\sum_{a\in A} u(a) a\right) \left(\sum_{b\in B} v(b) b\right)
\left(\sum_{c\in C} w(c) c\right)
$$
in $\mathbb{F}_p[G]$. Our goal is to find functions $u,v,w$ so that all
coefficients of the elements of $G$ in this product are non-zero. This clearly implies $ABC=G$. 

Assume that the first multiple, being expressed in powers of
$\tau$, has coefficient 1 of $\tau^{|A|-1}$
and coefficients 0 of $\tau^{i}$ for $i=0,\dots,|A|-2$. Analogous conditions are
imposed on two other multiples. Then our product equals 
$$
\tau^{p-1}=\sum_{i=0}^{p-1} (-1)^i g_0^{p-1-i},
$$
all coefficients are non-zero as desired. It remains to understand why we may choose
function, say, $u$ satisfying these conditions. Denote $A=\{(\tau+1)^{\alpha_1},\dots,(\tau+1)^{\alpha_{|A|}}\}$,
where 
$0\leqslant \alpha_1<\dots<\alpha_{|A|}\leqslant p-1$, then finding appropriate coefficients
$u(a)$ is solving a linear system. A matrix of this system has entries 
$\binom{\alpha_i}{j}$,
$1\leqslant i\leqslant |A|, 0\leqslant j\leqslant |A|-1$. It is non-degenerated
as a generalized Vandermonde type matrix.

These two proofs look like the same thing
said on different languages. But if we change the problem a bit,
say, consider restricted product sets $\{abc| a\in A,b\in B,c\in C,a\ne b\}$ (Erd\"os--Heilbronn
problem), then the first proof generalizes easily (we should replace polynomial to $(x-y)((x+y+z-h)^{p-1}-1)$).
And what about the second?
On the other hand, what are ``polynomial'' analogues of group rings 
proofs for groups not embeddable to fields? In some important
cases the modular
properties of binomial polynomials $\binom{x}n$ do work
(see \cite{DSB}), but what in general?
This parallelism is still unclear for me. 

Now we explain a group ring version of Croot--Lev--Pach ideas. For 
making the exposition
self-contained and comparing two technologies easier for a reader
we also include the polynomial proof for the bounds in the groups
$\mathbb{F}_p^n$.

\section{Linear algebraic lemmata and polynomial approach}

Let $\mathbb{F}$ be a field, $(A,\leqslant)$ be a linearly ordered finite set of size $|A|=d$. Denote
by $\mathbb{F}^A$ the $d$-dimensional space of functions $f:A\rightarrow {\mathbb F}$.
For non-zero element $z\in \mathbb{F}^A$ we define the \emph{leader} $\ell(z)$ and
the \emph{outsider} $out(z)$ as the minimal, corr. the maximal, $a\in A$ such that
$z(a)\ne 0$.

\begin{lemma}\label{Gauss} Let $W\subset \mathbb{F}^A$ be a linear subspace. 
Then there are exactly $\dim W$ different
leaders of the elements of $W$ (and exactly $\dim W$ different outsiders). 
\end{lemma}

\begin{proof} In two words, the proof is Gauss elimination. Here go the details.
Let $m$ denote the number of different leaders of the elements
of $W$;  $w_1,\ldots,w_m\in W$ be the elements with distinct
leaders $i_1<i_2<\ldots<i_m\in A$. Then the leader of the linear combination
$c_1w_1+\ldots +c_mw_m$, $c_i\in \mathbb{F}$, is $\min\{i_s: c_s\ne 0\}$. 
In particular, $w_1,\ldots,w_m$ are linearly independent and $\dim W\geqslant m$.
If this inequality is strict, there exists $w\in W$ such that $w$ is not a linear combination
of $w_1,\ldots,w_m$. Choose such $w$ with maximal possible leader $i_s$.
For certain scalar $c\in \mathbb{F}$ the element $u:=w-cw_s\in W$ takes zero value
at $i_s$, thus $u\ne 0$ and $\ell(u)>i_s$, a contradiction. So $\dim W=m$, Lemma is proved. 
\end{proof}

As an application of Lemma \ref{Gauss} we prove the following lemma
(the statement is motivated by the answer given by Ilya Bogdanov
on a similar question on Mathoverflow \cite{MO}),
useful for bounding the set without 3-progressions.

\begin{lemma}\label{codims}
Let $\mathbb{F}$ be a field, $A$ be a finite set. Denote
by $\mathbb{F}^A$ the $d$-dimensional space of functions $f:A\rightarrow {\mathbb F}$.
Let $X_1,\dots,X_k$ be linear subspaces of $\mathbb{F}^A$. 
Assume that $$\sum_{a\in A} f_1(a)f_2(a)\ldots f_k(a)=0\, \text{for all}\, f_i\in X_i.$$
Partition the set of indices $\{1,2,\ldots,k\}$ onto two non-empty disjoint subsets
$I_1$ and $I_2$. Then
\begin{equation}\label{codim-eq}
|A|\leqslant \codim \cap_{i\in I_1} X_i+\codim \cap_{i\in I_2} X_i.
\end{equation}
\end{lemma}

\begin{proof} Assume the contrary. Denote $Y_1=\cap_{i\in I_1} X_i, Y_2=\cap_{i\in I_2} X_i$.
Then $$\dim Y_1+\dim Y_2=2|A|-\left(\codim Y_1+\codim Y_2\right)>|A|$$
and applying Lemma \ref{Gauss} and pigeonhole principle we may find $b\in A$ which is simultaneously 
a leader of some function $g_1\in Y_1$ and outsider of some function $g_2\in Y_2$.
Choose $f_i=g_1$ for $i\in Y_1$ and $f_i=g_2$ for $i\in Y_2$. 
Then for $a<b$, correspondingly $a>b$, we have $f_i(a)=0$ for all $i\in I_1$, correspondingly for
all $i\in I_2$. Therefore
$$\sum_{a\in A} f_1(a)\ldots f_k(a)=f_1(b)\ldots f_k(b)\ne 0,$$
a contradiction.
\end{proof}

\begin{corollary}
If all $X_1,\ldots,X_k$ in Lemma \ref{codims} coincide, then $|A|\leqslant 2\codim X_1$.
Otherwise $|A|$ does not exceed the sum of codimensions of all \emph{distinct}
spaces in the collection $\{X_1,X_2,\ldots,X_k\}$.
\end{corollary}

\begin{proof} We apply \eqref{codim-eq} for a partition $\{1,\ldots,k\}=I_1\sqcup I_2$ onto two non-empty subsets
which we specify bit later.
Applying the estimate $\codim Y\cap Z\leqslant \codim Y+\codim Z$
we see that $ \codim \cap_{i\in I_1} X_i$ does not exceed the sum of codimensions
of all distinct subspaces in the family $\{X_i,i\in I_1\}$, analogously for $I_2$. 
If all $X_i$'s coincide, any partition gives the bound $|A|\leqslant 2\codim X_1$.
If not all $X_i$'s coincide, the partition may be chosen so that $X_i\ne X_j$ for all $i\in I_1,j\in I_2$ (that is, all indices with coinciding
subspaces belong to the same part). 
\end{proof}

Before formulating the group rings approach we
give a argument which allows to get 
exponential bounds
for 3-progression-free subsets in the additive group of $\mathbb{F}_p^n$ for odd 
prime power $p$.
It is essentially the same as Ellenberg--Gijswijt proof \cite{EG}
and is based on the ideas of Croot--Lev--Pach \cite{CLP}. 

For positive integers $p$, $n$ and non-negative integer $r$ denote
$$
T(n,p,r)=\left\{ (c_1,\ldots,c_n)\in \{0,1,\ldots,p-1\}^n:\sum_{i=1}^n c_i\leqslant r \right\}.
$$ 

\begin{theorem}
Let $p$ be a prime, $k\geqslant 2$ be an integer, $n_1,\ldots,n_k$
be non-zero elements of $\mathbb{F}_p$ such that $n_1+\ldots+n_k=0$. 

Assume that  $A\subset \mathbb{F}_p^n$ is a subset such that if 
$n_1g_1+\dots+n_kg_k=0$ for $g_i\in A$, then $g_1=\dots=g_k$.
Denote $(p-1)n=rk+(k-1-\theta)$, where $r=\lfloor (p-1)n/k\rfloor$, $0\leqslant \theta\leqslant k-1$. 
Then $$|A|\leqslant |T(n,p,r)|+|T(n,p,r-\theta)|.$$ 
\end{theorem}

\begin{proof} Denote $r_i=r$ for $i=1,2,\ldots,k-1$, $r_k=r-\theta$.
Consider the following
subspaces $X_1,\dots,X_k$ of the space  $\mathbb{F}_p^A$  of functions on $A$:
$X_i=\{f:\sum_{a\in A} f(a) h(a)=0\}$ for any polynomial $h(x_1,\ldots,x_n)$ of degree at most $r_i$. 
The condition  $f\in X_i$ is equivalent to the system of $ |T(n,p,r_i)$ equations
$\sum_{a\in A} f(a) h(a)=0$ for monomials $h(x)=x_1^{c_1}\ldots x_n^{c_n}$, where $(c_1,\ldots,c_n)\in T(n,p,r_i)$.
Therefore $\codim X_i\leqslant |T(n,p,r_i)|$.

Assume that $f_i\in X_i$. Let $F(x_1,x_2,\dots,x_k)$ be any polynomial
of $kn$ variables (here $x_i$ has $n$ coordinates $x_{i1},\dots,x_{in}$) such that $\deg F\leqslant (p-1)n$. Then 
\begin{equation}\label{poly_main}
\sum_{a_i\in A} f_1(a_1)\dots f_k(a_k)F(a_1,\dots,a_k)=0.
\end{equation}
 Indeed, for any monomial $G(x_1,\ldots,x_k)$ of 
the polynomial $F$ we have $$\sum_{i=1}^k \deg_{x_i} G=\deg G\leqslant (p-1)n=rk-\theta+k-1=(r_1+1)+\ldots+(r_k+1)-1,$$
hence by pigeonhole principle there exists $i$ such that 
$\deg_{x_i} G\leqslant r_i$. It yields that for fixed $a_1,a_2,\ldots,a_{i-1},a_{i+1},\ldots,a_k$
we have
$$
\sum_{a_i\in A} f_i(a_i) G(a_1,\ldots,a_k)=0,
$$
due to $f_i\in X_i$. Therefore the whole sum
$\sum_{a_i\in A} f_1(a_1)\dots f_k(a_k)G(a_1,\dots,a_k)$ is also zero. 

Apply \eqref{poly_main} to the polynomial 
$$
F(x_1,\dots,x_k)=\prod_{j=1}^n \left(1-\left(\sum_{i=1}^k n_ix_{ij}\right)^{p-1}\right).
$$
We get 
$$
0=\sum_{a_i\in A} f_1(a_1)\dots f_k(a_k)F(a_1,\dots,a_k)=\sum_{a\in A} f_1(a)\dots f_k(a),
$$
the last identity follows from the assumption
that $A$ does not contain non-trivial solutions of $\sum n_ia_i=0$.
Now Lemma \ref{codims} implies that 
$$|A|\leqslant \codim X_1+\codim X_r\leqslant |T(n,p,r)|+|T(n,p,r-\theta)|.$$
\end{proof}

When $k=3,n_1=n_2=1,n_3=-2$, we get the following

\begin{corollary} If $p$ is odd prime and $A\subset \mathbb{F}_p^n$ does not contain
proper 3-progressions, then $$|A|\leqslant T(n,p,r)+T(n,p,r-\theta),$$
where $r=\lfloor (p-1)n/3\rfloor$, $(p-1)n=3r+(2-\theta)$.
\end{corollary}

\section{Group rings approach}

Now let $G$ be a finite group (not necessarily abelian), $n_1,\dots,n_k$ (where $k\geqslant 2$)
be non-zero integers which sum up to $0=\sum n_i$. Assume that 
all $n_i$ are coprime to $|G|$. 
Let a subset $A\subset G$ be so that the 
equation $g_1^{n_1}\dots g_k^{n_k}=1$, $g_i\in A$, holds only for
$g_1=\dots =g_{k}$. Arithmetic progressions of length
3 correspond to the case $k=3$, $n_1=n_2=1$, $n_3=-2$. It is probably more
natural to call a solution of $g_1g_2^{-1}g_3g_2^{-1}=1$ ``an arithmetic progression''
in non-abelian setting,
but, alas, I do not know how to modify the argument for such equations.

Assume that for some field $\mathbb{F}$ we managed to find subspaces $X_1,\dots,X_k$ of a group algebra
$\mathbb{F}[G]$ satisfying $X_1\dots X_k=0$,
i.e., $u_1\dots u_k=0$ for $u_i\in X_i$. Denote $t_i=\codim X_i$.
A general fact is the following

\begin{theorem}\label{equation_gr} In above assumptions
call two indices $i,j\in\{1,\dots,k\}$ \emph{equivalent}
if $X_i=X_j$ and $n_i=n_j$. If $I$ is a maximal system of mutually non-equivalent indices, 
then 
$$
|A|\leqslant \begin{cases}\sum_{i\in I} t_i& \text{if}\,\, |I|\geq 2\\ 2t_1& \text{if}\,\, |I|=1.\end{cases}
$$
In particular, we always have $|A|\leqslant \sum t_i$. 
\end{theorem}

\begin{proof} For $g\in G$, $u\in \mathbb{F}[G]$, let $[g]u$ denote a coefficient of $g$ in $u$.
Denote $A^{r}=\{a^r,a\in A\}$ for $r=1,2,\dots$. Then $|A^{r}|=|A|$ if
$r$ and $|G|$ are coprime. Let further $W_i$ denote 
the span of $A^{n_i}$ in $\mathbb{F}[G]$, we have $\dim W_i=|A|$.
Finally denote $L_i=X_i\cap W_i$, we have $\dim L_i\geqslant \dim X_i+\dim W_i-|G|=|A|-t_i$.
Consider functions $f_1,\dots,f_k$ on $A$ such that $\sum_{a\in A} f_i(a)a^{n_i}\in L_i$.
Then
$$
\sum_{a\in A} f_1(a)\dots f_k(a)=[1] \left(\sum_{a\in A} f_1(a) a^{n_1}\right)\cdot \left(\sum_{a\in A} f_2(a) a^{n_2}\right)\cdot \ldots \cdot 
\left(\sum_{a\in A} f_k(a) a^{n_k}\right)=0,
$$
the first equality follows from our assumption on $A$. 
It remains to use Corollary of Lemma \eqref{codims}.
\end{proof}

\begin{remark} If $k=2$, then even $A=G$ satisfies the condition of the Theorem. So, what we actually get in this case is
that if $X_1X_2=0$ for two subspaces of $\mathbb{F}[G]$, then $\codim X_1+\codim X_2\geqslant |G|$ (that is, the
result is about group rings, not about combinatorics). But already for 3 multiples for some groups there exist subspaces
$X_1,X_2,X_3$ of low codimension (low means $o(|G|)$ or sometimes even $O(|G|^{c})$, $c<1$) such that $X_1X_2X_3=0$.
The examples of such groups are given below.
\end{remark}

\section{Kleinberg--Sawin--Speyer refinement}

The set $A\subset G$ without arithmetic progressions of length 3
produces a set of triples $\{(x_a,y_a,z_a):=(a,a,a^{-2}),a\in A\}\subset G^3$
such that $x_ay_bz_c=1$ if and only if $a=b=c$. Following \cite{KSS} we call such a set of ordered
triples a ``multiplicative 3-matching''. If the group is additive we say ``additive 3-matching'' instead.
 Analogously we may defined multiplicative (and additive) $k$-matchings
of $k$-tuples. Note that the above proof for Theorem \ref{equation_gr} allows to show the following

\begin{theorem} Assume that the subspaces $X_1,X_2,\dots,X_k$ of a group algebra
$\mathbb{F}[G]$ satisfy $X_1\dots X_k=0$.
A multiplicative $k$-matching may contain at most $\sum t_i$ $k$-tuples, where $t_i=\codim X_i$.
\end{theorem}

Remark that for $k=3$ we get the estimate $3\codim X$ (where $X^3=0$, $X\subset \mathbb{F}[G]$),
for the multiplicative 3-matchings, but $2\codim X$ for the sets without non-trivial solutions of
$xy=z^2$ (provided that $|G|$ is odd.) This improvement comes from the coincidence of two
spaces of functions in the latter case. 

What is remarkable is that the estimate obtained
on this way for multiplicative 3-matchings of $\mathbb{F}_p^n$
is asymptotically tight (in logarithmic scale), this follows from the results of \cite{KSS} and \cite{No,Pe}.
It would be nice to get the tightness for other groups. Let me formulate it as a conjecture:

\begin{conjecture} Consider all $k$-tuples of subspaces $X_1,\dots,X_k$ of
$\mathbb{F}[G]$ such that $X_1\dots X_k=0$. Take the minimal value
$M_k(G)$ of the sum of their codimensions. (Here the minimum is taken over the choice
of the field also.) Then, for large $|G|$, the maximal cardinality of a multiplicative $k$-matching of
$G$ is $M_k(G)\cdot |G|^{o(1)}$. 
\end{conjecture}

This may seem too strong. At first, we actually use in the proof only that $[1]X_1\dots X_k=0$,
not $X_1\dots X_k=0$. But I do not know whether it allows to improve the bound. 
At second, the minimizing over $\mathbb{F}$ looks somehow strange. Possibly,
we should start from the case of $p$-groups and the group rings over $\mathbb{F}_p$.

\section{Ellenberg's refinement}

In \cite{El} the following nice further generalization 
is proved: for any sets $A$, $B$ in the additive group of
$\mathbb{F}_p^n$ there exist subsets $A_1\subset B$, $B_1\subset B$
for which $A+B\subset (A_1+B)\cup (A+B_1)$ and $|A_1|+|B_1|\leqslant M$. Here
$M$ is the \cite{EG} and \cite{KSS} bound $3T(n,p,\lfloor n(p-1)/3 \rfloor)$.

Note that this implies the upper bound $M$ for the size of additive 3-matching. Indeed, if $\{(x_i,y_i,z_i):i=1,2,\ldots,m\}$
is an additive 3-matching in $\mathbb{F}_p^n$, then denote $A=\{x_1,\ldots,x_m\}$,
$B=\{y_1,\ldots,y_m\}$. If $A+B\subset (A_1+B)\cup (A+B_1)$, then for any $i$
we have either $x_i\in A_1$ or $y_i\in B_1$, otherwise $-z_i=x_i+y_i$ belongs to $A_1+B$
or $A+B_1$, that is, may be represented as $x_j+y_k$
for $j$ or $k$ different from $i$. This contradicts to the definition of additive 3-matching. Therefore
$m\leqslant |A_1|+|B_1|$.

Here is the abstract result in spirit of \cite{El}:

\begin{theorem}\label{ELL} Let $G$ be a finite group, and $X_0,X_1,\dots,X_k$
be linear subspaces of the group algebra $\mathbb{F}[G]$ such that
$X_0 X_1 \ldots X_k=0$. Denote $t_i=\codim X_i$.
Let $A_1,\dots,A_{k}$ be arbitrary subsets of $G$. Then there
exist subsets $B_i\subset A_i$, $i=1,\dots,k$, and $C\subset G$ 
such that $|C|\leqslant t_0$, $|B_i|\leqslant t_i$ for all $i=1,\dots,k$, and
\begin{equation}\label{Ell-covering}
A_1A_2\cdot \ldots \cdot A_k\subset C\cup 
B_1A_2\cdot \ldots \cdot A_k\cup A_1B_2\cdot \ldots \cdot A_k
\cup \ldots \cup A_1A_2\cdot \ldots \cdot A_{k-1}B_k.
\end{equation}
Moreover, we may additionally require that $B_i=B_j$ ($1\leqslant i<j\leqslant k$) 
whenever 
simultaneously $X_i=X_j$ and $A_i=A_j$.
\end{theorem}

Ellenberg uses the result of Roy Meshulam \cite{Mes};
we do not. But the concept of lexicographic leaders of 
multi-linear forms, used by Meshulam, is still important for us.

\begin{proof}[Proof of Theorem \ref{ELL}]
Denote $A_0=(A_1\cdot A_2\cdot \ldots \cdot A_k)^{-1}$.
Fix arbitrary linear orderings on $A_1,\ldots,A_k$. 
Fix also a linear extension of the direct product of these
orderings on $A_1\times A_2\times \ldots\times A_k$
(for example, the lexicographical order). Finally,
for any $g\in A_0$ take the minimal $k$-tuple
$(a_1,\ldots,a_k)\in A_1\times \ldots \times A_k$ for which
$g=(a_1\ldots a_k)^{-1}$, this map from $A_0$ to $A_1\times \ldots \times A_k$
induces the order on $A_0$. 

Consider the linear subspace $W_0=\mathbb{F}[A_0]\cap X_0$.
The codimension of
$W_0$ in $ \mathbb{F}[A_0]$ does not exceed $t_0$.
Denote by $C^{-1}\subset A_0$ the set of the elements of $A_0$ which \emph{are not}
leaders of the elements of $W_0$. 
For $i=1,\ldots,k$ denote by $B_i\subset A_i$ the set of elements 
of $A_i$ which \emph{are not} outsiders of the elements of the subspace $\mathbb{F}[A_i]\cap X_i$.
By Lemma \ref{Gauss} we get $|C|=|C^{-1}|\leqslant t_0$, $|B_i|\leqslant t_i$ for $i=1,\ldots,k$.

Assume that \eqref{Ell-covering} does not hold. Then there exists
$g^{-1}\in A_0^{-1}=A_1A_2\ldots A_k$ not covered by the sets $C$,
$B_1A_2\ldots A_k$, $A_1B_2\cdot \ldots \cdot A_k$, $\ldots$, $A_1A_2\cdot \ldots \cdot A_{k-1}B_k$.
Since $g^{-1}\notin C$, we get $g\notin C^{-1}$, hence $g$ is the leader of a certain 
element $f_0\in W_0$. Let $(a_1,a_2,\ldots,a_k)\in A_1\times A_2\times \ldots \times A_k$ be
the minimal sequence with $a_1\ldots a_k=g^{-1}$. If $a_i\in B_i$ for some $i$, then $g^{-1}$
is covered by corresponding set $A_1\ldots A_{i-1}B_iA_{i+1}\ldots A_k$, that contradicts to our
assumption.  So $a_i\notin B_i$ for all $i=1,\ldots,k$, and there exist
the elements $f_i\in \mathbb{F}[A_i]\cap X_i$ with outsiders $out(f_i)=a_i$.
The product
$f_0f_1f_2\ldots f_k$ equals 0, since it belongs to $X_0X_1\ldots X_k$.  
But expand the brackets and consider the coefficient of 1 in this product $f_0f_1f_2\ldots f_k$. 
1 appears with non-zero coefficient if we take $g$ from $f_0$ and $a_i$ from $f_i$, $i=1,2,\ldots,k$. 
Hence for cancellation  it must appear by some other way too,
and we have $hb_1b_2\ldots b_k=1$ for certain $h\geqslant g$, $h\in A_0$ (since $g=\ell(f_0)$) 
and $b_i\leqslant a_i$ (since $a_i=out(f_i)$), $i=1,\ldots,k$. But it yields $(b_1,\ldots,b_k)\leqslant (a_1,\ldots,a_k)$,
since the order on $A_1\times A_2\times \ldots \times A_k$  is a linear extension of the product partial
order. Therefore $h\leqslant g$ by the definition of order on $A_0$, and so all inequalities are equalities: $h=g$,
$b_i=a_i$ for $i=1,\ldots,k$. We see that this way to get 1 is the same as before, a contradiction.
\end{proof}

Note that Theorem \ref{equation_gr} follows from Theorem \ref{ELL} in many
cases (for example, for 3-progressions), but there are cases when the estimate 
of Theorem \ref{equation_gr} is bit better (for example, if $n_1=n_2,n_3=n_4$,
$X_1=X_2,X_3=X_4$ we get $|A|\leqslant t_1+t_2$ using Theorem
\ref{equation_gr} , but only $|A|\leqslant 2t_1+t_2$ using Theorem \ref{ELL}. )

Below we discuss how to find the huge subspaces of group algebras with zero product in several situations.

\section{Abelian $p$-groups}

Let $p$ be a prime.
Let $G=\prod_{i=1}^n C_{N_i}$ be a finite Abelian $p$-group with $n$ generators $g_1,\dots,g_n$:
$g_i$ generates $C_{N_i}$, and each $N_i$ is a power of $p$. 

The group algebra $\mathbb{F}_p[G]$ is generated by the products
$\prod (1-g_i)^{m_i}$, where $m_i\in \{0,1,\dots,N_i~-~1\}$. 
Let $\eta_1,\dots,\eta_k$ be positive reals such that $\sum \eta_i=1$ (usually the choice
$\eta_i=1/k$ is optimal). Fix also positive parameters $\lambda_1,\dots,\lambda_n$.
Let $X_i$, $i=1,\dots,k$, be a subspace generated by
monomials for which 
$$
\sum_{j=1}^n \lambda_j \left(\frac{m_j}{N_j-1}-\eta_i\right)>0
$$ 
Any product $f_1\dots f_k$ 
for $f_i\in X_i$ has some guy $(1-g_j)$ in a power strictly greater than $N_j-1$, but $(1-g_j)^{N_j}=0$.

In order to estimate the codimension $t_i$ of $X_i$ we may use Chernoff bound (which is essentially tight for large $n$ due to
Cramer theorem). Take $\eta_1=\dots =\eta_k=1/k$, then $X_1=\dots=X_k$, $t_1=\dots=t_k:=t$.
If $\xi_j$ are random variables 
uniformly distributed in the set $\frac1{N_j-1}\{0,1,\dots,N_j-1\}$, then 
$t/|G|$ is the probability that $\sum_j \lambda_j(\xi_j-\frac1k)\leqslant 0$.
If this inequality holds, for any $x\in (0,1)$ we have
$$
x^{\sum \lambda_j \xi_j}\geqslant x^{(\lambda_1+\dots+\lambda_n)/k}.
$$
For $N\geqslant 2$ denote
$$
S_N(x):=\frac 1{N} \sum_{j=0}^{N-1} x^{j/(N-1)}=\mathbb{E} \left\{x^\xi, \xi\,\,\text{uniformly distributed on\,\,} \frac1{N-1}\{0,1,\dots,N-1\}\right\}.
$$
Then Chebyshev inequality implies that
$$
t/|G|\leqslant \prod_{j=1}^n x^{-\lambda_j/k} S_{N_j}(x^{\lambda_j}).
$$
Now we fix the value of $x$ and parameters $\lambda_1,\dots,\lambda_n$ so that this rewrites as
$$
t/|G|\leqslant \prod_{j=1}^n \min_{x\in (0,1]} x^{-1/k}S_{N_j}(x).
$$
In particular, for $k=3$, where we denoted above $N\cdot \min_{x\in (0,1]} x^{-1/k}S_{N}(x)=\kappa_N$, we get
the following 

\begin{theorem}\label{Abelian} If all $N_i$ are powers of the same odd
prime $p$, then the size of a set $A\subset \prod C_{N_i}$ without 3-progressions is at most
$\prod_i {\kappa_{N_i}}$.
\end{theorem}

We have $S_2(x)>S_3(x)>S_4(x)>\dots$ for any positive $x\ne 1$
(this follows from Karamata majorization inequality, for example), and $\lim_n S_n(x)=(x-1)/\log x$. Thus the sequence $a_N(k)=\min_{x\in (0,1]} x^{-1/k}S_{N}(x)$
decreases with $N$. For example, if $k=3$, it varies from $a_2(3)=0.9449\dots$ to $a_{\infty}(3)=0.8414\dots$.
It means that increasing $N_j$ makes our bound for
$t/|G|$ better. Say, for $C_9^n$ we get exponentially better bounds on $|A|$ than for $3^n$ copies of $C_3^n$
(onto which we could partition $C_9^n$ in order to get some exponential bound): if $k=3$, improvement is 
$t/|G|\leqslant 0.872^n$ compared to $t/|G|\leqslant 0.919^n$.

We may also 
rephrase the original result by Croot, Lev and Pach for $G=\mathbb{C_4}^n$ and solutions
of the equation $xyz^{-2}=1$ with mutually different $x,y,z$
in the same spirit (and get the same bound as they get). The difference with the situation considered above
is that exponent 2 is not coprime
with $|G|$, so $A^2$ is in general situation much less than $A$, and we do not count $x^2=y^2$, $x\ne y$, as a non-trivial solution. 
This is handled by a partition onto 
classes modulo $G^2$ and writing such products for $a,b,c$ in the same class, as 
Croot, Lev and Pach do. This does not work
as easy for, say, $C_8$, since the kernel and image of the homomorphism
$g\rightarrow g^2$ no longer coincide. I suppose that Theorem \ref{Abelian}
should hold also for general 2-groups. 

Another proofs of Theorem \ref{Abelian} are proposed by Will Sawin and Eric Naslund (they use
divisibility of binomial coefficients) and by David Speyer (he uses Witt vectors), see the explanation
in \cite{DSB}.

\section{Matrix $p$-groups}

Heisenberg group $G=H_{n-1}(\mathbb{F}_p)$ of order $p^{2n-1}$ consists
of $(n+1)\times (n+1)$ matrices $(a_{ij})_{0\leqslant i,j\leqslant n}$ 
over $\mathbb{F}_p$ satisfying
conditions $a_{ii}=1$; $a_{ij}=0$ unless $i=j$ or $i=0$ or $j=n$:
$$
\begin{pmatrix} 1&a_{01}&\hdotsfor{2}&a_{0n}&a_{0n}\\
0&1&0&0&\hdotsfor{1}&a_{1n}\\
0&0&1&0&\hdotsfor{1}&a_{2n}\\
\vdots& \vdots & \vdots &\ddots& \ddots&\vdots \\
0&0&0&\hdotsfor{2}&1
\end{pmatrix}.
$$
It is generated by the elements $g_1,\dots,g_{n-1},h_1,\dots,h_{n-1},s$
where $g_i=id+e_{0,i}$, $h_i=id+e_{i,n}$, $s=id+e_{0,n}$, where $id$ denotes identity
matrix, $e_{i,j}$ denote matrix units. They satisfy relations $g_ih_{j}=h_ig_j$
unless $i=j$, $g_ih_i=sh_ig_i$, $s$ commutes with all $g_i$ and $h_i$. 

We may uniquely write each element in the form 
$s^{\gamma}h_1^{\beta_1}\dots h_{n-1}^{\beta_{n-1}}g_1^{\alpha_1}\dots g_{n-1}^{\alpha_{n-1}}$,
$0\leqslant \gamma,\beta_i,\alpha_i\leqslant p-1$. 
Denote the following elements of $\mathbb{F}_p[G]$: $z=s-1$, $y_i=h_i-1$, $x_i=g_i-1$.
Then $z^p=y_i^p=x_i^p=0$. Elements of the form
$z^{\gamma}y_1^{\beta_1}\dots y_{n-1}^{\beta_{n-1}}x_1^{\alpha_1}\dots x_{n-1}^{\alpha_{n-1}}$ are called reduced.
Reduced elements for which
$0\leqslant \gamma,\beta_i,\alpha_i\leqslant p-1$ form a basis of $\mathbb{F}_p[G]$. Our
relations read as as $x_iy_i=y_ix_i+zy_ix_i+z+zx_i+zy_i$. Define
a degree of an arbitrary word (sequence) in the alphabet $\{z,x_1,\dots,x_{n-1},y_1,\dots,y_{n-1}\}$
as twice number of $z$'s plus number of all other letters. Key lemma is that any word of
degree $D>2n(p-1)$ equals 0 being evaluated as a group ring element. 
Indeed, replacing $x_iy_i=y_ix_i+zy_ix_i+z+zx_i+zy_i$ and using other commutativity relations we may reduce
each word to a sum of reduced words of the same or greater degree. Any such a word equals 0
provided that its degree exceeds $2n(p-1)$. Now define a subspace $X$ of 
$\mathbb{F}_p[G]$ formed by reduced monomials of degree strictly greater than $2n(p-1)/k$. 
We have $X^k=0$ and codimension of $X$ is exponentially small for large $n$ (if $k\geqslant 3$)
by essentially the same reasons
as in previous section. 

The group $G=UT(n,\mathbb{F}_p)$, $|G|=p^{n(n-1)/2}$, of upper unitriangular $n\times n$-matrices
over $\mathbb{F}_p$ may be treated similarly. 
Namely, we may choose generators $g_{ij}=id+e_{ij}$,
$i<j$ in $G$; each element of $G$ has unique representation as a product of these
generators taken in inverse lexicographic order and in powers at most $p-1$
$$
g_{n-1,n}^{\alpha_{n-1,n}} g_{n-2,n}^{\alpha_{n-2,n}}\dots g_{1,2}^{\alpha_{1,2}}, 0\leqslant \alpha_{i,j}\leqslant p-1.
$$
$g_{ij}$ and $g_{kl}$ commute unless $j=k$ or $i=l$. In this case we have relations
$g_{ij}g_{jl}=g_{jl}g_{ij}g_{il}$. Denote $x_{ij}=g_{ij}-1$ in a group ring $\mathbb{F}_p[G]$, we have $x_{ij}^p=0$ and
there is a basis in $\mathbb{F}_p[G]$
formed by the elements 
$$
x_{n-1,n}^{\alpha_{n-1,n}} x_{n-2,n}^{\alpha_{n-2,n}}\dots x_{1,2}^{\alpha_{1,2}}, 0\leqslant \alpha_{i,j}\leqslant p-1.
$$
Define a degree of any word in alphabet $\{x_{ij}\}$'s as a sum of $(j-i)$ over all used letters
(multiplicity counted of course). 
If $i<j<l$ we have $(1+x_{ij})(1+x_{jl})=(1+x_{jl})(1+x_{ij})(1+x_{il})$, thus
$x_{ij}x_{jl}=x_{jl}x_{ij}+x_{il}+x_{ij}x_{il}+x_{jl}x_{il}+x_{jl}x_{ij}x_{il}$. Using these relations
we may reduce each word to a sum of reduced words of the same or greater degree. It remains to define
define a subspace $X$ of 
$\mathbb{F}_p[G]$ formed by reduced monomials of degree strictly greater than $(p-1)(\sum_{i<j}(j-i))/k=(p-1)(n^3-n)/(6k)$.

In particular above constructions (combined with Chernoff bound) prove the following 

\begin{theorem}
For any prime $p$ there exists $\lambda>0$ depending on $p$ such that the group rings $\mathbb{F}_p[G]$, where $G=H_{n-1}(\mathbb{F}_p)$ is
a Heisenberg group over $\mathbb{F}_p$ or 
or $G=UT(n,\mathbb{F}_p)$ is a group of upper triangular matrices over $\mathbb{F}_p$, there exists a subspace $X\subset \mathbb{F}_p[G]$
such that $X^3=0$ and $\codim X<|G|^{1-\lambda}$. 
\end{theorem}

I suppose that the huge subspaces with zero cube must exist in the group rings
of all finite groups with small exponent:

\begin{conjecture} For any $N$ there exists 
$\lambda<1$ such that any group
$G$ in which $g^N=1$ for all $g\in G$, the group algebra $K[G]$ has a subspace
with zero cube of codimension at most $|G|^\lambda$. 
\end{conjecture}

Here $K$ may depend on $G$,
I guess that $K=\mathbb{F}_p$ for some prime divisor $p$ of $N$ should work.

Even if true, this does not cover the upper unitriangular matrices case, so
possibly something even better holds.

\section{Acknowledgments}

I am grateful to Vsevolod Lev, Jordan Ellenberg, Ilya Bogdanov, Ilya Shkredov, Roman Karasev,
Roman Mikhailov, Cosmin Pohoata, 
Alexander Efimovich Zalesskii and Anatoly Moiseevich Vershik for fruitful discussions.

\end{document}